\documentclass[11pt]{amsart}
 
\usepackage{mathtools}
\mathtoolsset{showonlyrefs,showmanualtags}
 
\usepackage{amssymb, amsmath, amsthm, esint}
\usepackage{mathrsfs}
\usepackage{bbm, dsfont}
\usepackage{color}
\usepackage{graphicx}

\usepackage[top=1.5in, bottom=1.5in, left=1.25in, right=1.25in]{geometry}
\usepackage[all]{xy}
\usepackage{amscd}
 \usepackage{microtype}
\usepackage[bottom]{footmisc}
 
\usepackage{fancyhdr}
 
\AtBeginDocument{\setlength{\footskip}{24pt}}
\fancypagestyle{plain}{%
  \fancyhf{}%
  \fancyfoot[C]{\thepage}%
}
 
\fancypagestyle{myfancy}{%
  \fancyhf{}%
  \fancyhead[CE]{\scshape Pointwise Convergence Along Integer Cantor Sets}%
  \fancyhead[CO]{F.\ Brokering Pinilla, A.\ Iosevich and B.\ Krause}%
  \fancyhead[LE,RO]{\thepage}%
}
 
\numberwithin{equation}{section}

\newtheorem{theorem}{Theorem}[section]
\newtheorem*{theorem*}{Theorem}

\newtheorem{proposition}[theorem]{Proposition}
\newtheorem*{proposition*}{Proposition}
\newtheorem{cor}[theorem]{Corollary}
\newtheorem{lemma}[theorem]{Lemma}
\newtheorem*{lemma*}{Lemma}
\newtheorem{remark}[theorem]{Remark}

\theoremstyle{definition}

\usepackage[
  pagebackref,
  colorlinks=true,
  linkcolor=blue,
  citecolor=magenta,
  filecolor=magenta,
  urlcolor=cyan
]{hyperref}

\makeatletter
\renewcommand*{\backref}[1]{}
\renewcommand*{\backrefalt}[4]{%
  \ifcase #1\relax
  \or
    \space (Page~#2)%
  \else
    \space (Pages~#2)%
  \fi
}
\makeatother

\begin{document}
 
\title{Pointwise Convergence of Ergodic Averages Along Integer Cantor Sets}
 
\author{F\'elix Brokering Pinilla}
\address{
Department of Mathematics,
University of Bristol \\
Beacon House, Queens Rd, Bristol BS8 1QU}
\email{felix.brokeringpinilla@bristol.ac.uk}

\author{Alex Iosevich}
\address{%
Department of Mathematics, University of Rochester \\
14627, Rochester, NY, Hylan Building, 915}
\email{iosevich@gmail.com}

\author{Ben Krause}
\address{%
  Department of Mathematics, University of Bristol \\
  Beacon House, Queens Rd, Bristol BS8 1QU}
\email{ben.krause@bristol.ac.uk}

\date{\today}
 
\begin{abstract}
Let $d \geq 3$, 
  \[ D \subsetneq \{ 0,1,\dots,d-1\}, \qquad |D| \geq 2, \ 0 \in D \]
  be a finite alphabet, and define the integer Cantor set
  \begin{align}
\mathcal{C} := \mathcal{C}_{D} := \bigcup_{J \geq 0} \Big\{ \sum_{j =0}^J a_j d^j : a_j \in D \Big\}.   
  \end{align}
We prove that for any $\sigma$-finite measure-preserving system, $(X,\mu,T)$, and any $f \in L^p(X)$, $2\leq p<\infty$, the ergodic averages
\begin{align}
    \frac{1}{|\mathcal{C}_N|} \sum_{n \in \mathcal{C}_N } f(T^n x), \qquad \mathcal{C}_N := \mathcal{C} \cap \{1,2,\dots,N \}
\end{align}
  converge $\mu$-almost everywhere. By rescaling, this allows us to resolve the question of lacunary differentiation of Cantor measures at self-similar scales: if 
\begin{align}
    \mathcal{C}' := \Big\{ \sum_{j \geq 1} a_j d^{-j} : a_j \in D \Big\} \subset [0,1]
\end{align}
is a real-variable Cantor set, and $\nu$ denotes its natural measure, then we prove that
\begin{align}
    \lim_{k \to \infty} \int f(x-d^{-k} t) \ d\nu(t) = f(x)
\end{align}
Lebesgue almost-everywhere for any $f \in L^2_{\text{loc}}(\mathbb{R})$.
\end{abstract}
 
\maketitle
 
\setcounter{footnote}{0} 
 
\pagestyle{myfancy}
 

\section{Introduction}
The study of integer Cantor sets, i.e.\ those of the form
\begin{align}\label{e:cantor}
\mathcal{C} := \mathcal{C}_{D} := \bigcup_{J \geq 0} \Big\{ \sum_{j=0}^J a_j d^j : a_j \in D \Big\},
\end{align}
where
\[
D \subsetneq \{0,1,\dots,d-1\}, \qquad d \geq 3, \quad q:= |D| \geq 2, \quad 0 \in D,
\]
goes back at least to work of Fine \cite{Fine}, but has had many rich extensions; in recent years, Maynard established the existence of infinitely many primes inside these sets \cite{Maynard}, Leng--Sawhney established a ternary Goldbach-type result \cite{LS}, and Green demonstrated Waring-type phenomena \cite{Green}. Most relevant to our current context is \cite{B+}, which established the following dynamical theorem.

\begin{theorem}\label{t:lacey}
Let $(X,\mu)$ be a probability space equipped with a measure-preserving transformation $T:X\to X$. Then for any integer Cantor set of the form \eqref{e:cantor}, and any $f\in L^2(X)$, the averages
\begin{align}\label{e:avg}
C_N^T f(x)
:= \frac{1}{|\mathcal{C}_N|}\sum_{n\in\mathcal{C}_N} f(T^n x),
\qquad
\mathcal{C}_N:=\mathcal{C}\cap\{1,\dots,N\},
\end{align}
converge in $L^2(X)$.
\end{theorem}

On the other hand, the issue of pointwise convergence of the averages \eqref{e:avg} was unaddressed: see \cite[Question 6.3]{B+} and also \cite[Question 1.6]{KS}, which focused on the more general case of rational-spectra IP sets.

Our main result establishes this convergence in slightly greater generality; for a description of the limit in the finite-measure setting, see \cite[Theorem 2.1]{B++}, or \cite[Theorem 3.6]{KS} more generally.

\begin{theorem}\label{t:main0}
For any $\sigma$-finite measure space $(X,\mu)$ equipped with a measure-preserving transformation $T:X\to X$, and any $f\in L^p(X)$, $2\leq p<\infty$, the averages \eqref{e:avg} converge $\mu$-almost everywhere.
\end{theorem}

We prove Theorem \ref{t:main0} through re-parameterized averages: if $\Phi : \mathbb{Z}_{\geq 0} \to \mathcal{C}$ is the bijective digit-substitution map defined in \eqref{e:Phi}, we establish convergence of the averages
\begin{align}\label{e:ergodic-count-prefix}
A_M^T f(x):=\frac1M\sum_{m=0}^{M-1}f(T^{\Phi(m)}x).
\end{align}

In fact, we establish a quantitative version of Theorem \ref{t:main0}; to state it, recall that the homogeneous $r$-variation of a scalar sequence is
\begin{align}\label{e:variation-definition}
\mathcal V^r(a_n:n)
:=\sup_{n_0<\cdots<n_J}
\big(\sum_{j=1}^J|a_{n_j}-a_{n_{j-1}}|^r\big)^{1/r},
\qquad 0<r<\infty,
\end{align}
where the supremum runs over all finite increasing subsequences; note that these quantities decrease pointwise as $r$ increases. For a sequence of functions, variation is defined pointwise:
\begin{align}\label{e:fxnvar}
\mathcal V^r(f_n:n)(x):=\mathcal V^r(f_n(x):n);
\end{align}
since finite $r$-variation forces convergence of the scalar sequence, norm estimates on \eqref{e:fxnvar} are a powerful way to prove pointwise almost everywhere convergence.

Variational methods entered harmonic analysis through martingale convergence questions \cite{LE}, and have since become a central tool for quantifying pointwise convergence of operator families; see \cite{JSW,BK} for surveys. We establish Theorem \ref{t:main0} by way of the following; prior work of Bourgain \cite{B4} established a maximal estimate when $q = 2$, and identified the limit in the finite measure setting in this case as well.

\begin{theorem}\label{t:main}
For any $\sigma$-finite measure space $(X,\mu)$ equipped with a measure-preserving transformation $T:X\to X$, every integer $s\geq0$, and every $r>2$,
\begin{align}\label{e:quant-erg}
\Big\|\mathcal V^r(A_M^T f:M\in\mathcal G_s)\Big\|_{L^2(X)}
\lesssim_{d,D,s}\frac{r}{r-2}\|f\|_{L^2(X)},
\end{align}
where
\begin{align}\label{e:Gs}
\mathcal G_s:=\{uq^n:u,n\in\mathbb Z_{\geq0},\ q^s\leq u<q^{s+1}\},
\end{align}
and the times in $\mathcal G_s$ are read in increasing numerical order.\footnote{See Subsection \ref{sss:O} below for the definition of $\lesssim_{d,D,s}$.}
\end{theorem}

By Calder\'{o}n's transference principle \cite{C1}, Theorem \ref{t:main} is established by the following result concerning sequence-space functions.

\begin{theorem}\label{t:mainZ}
Let
\begin{align}\label{e:discrete-count-prefix-intro}
A_Mf(x):=\frac1M\sum_{m=0}^{M-1}f(x-\Phi(m)).
\end{align}
Then, for every integer $s\geq0$ and $r>2$,
\begin{align}\label{e:quant-Z}
\Big\|\mathcal V^r(A_Mf:M\in\mathcal G_s)\Big\|_{\ell^2(\mathbb Z)}
\lesssim_{d,D,s}\frac{r}{r-2}\|f\|_{\ell^2(\mathbb Z)}.
\end{align}
\end{theorem}

Theorem \ref{t:main0} is obtained from Theorem \ref{t:main}, a maximal estimate for \eqref{e:ergodic-count-prefix}, and the fact that a fixed mesh $\mathcal G_s$ approximates every sufficiently large integer $M$ up to a multiplicative factor of at most $1 +q^{-s}.$ The strictly increasing map $\Phi$ then identifies the averages $C_N^T$ with a subsequence of the re-parametrized averages, up to the harmless removal of the term $\Phi(0)=0$. 

A pleasant corollary is the following statement about lacunary differentiation of deterministic Cantor sets, complementing deep work of Laba--Pramanik \cite{LP} and Shmerkin--Suomala \cite{SS} on differentiation of random Cantor sets.

\begin{cor}\label{c:cantor}
Let
\begin{align}\label{e:cantor'}
\mathcal C':=\Big\{\sum_{j=1}^\infty a_jd^{-j}:a_j\in D\Big\}\subset[0,1]
\end{align}
denote a Euclidean Cantor set, and let $\nu$ denote its natural measure (whose law is given by \eqref{e:Y} below). Then, for every $f\in L^2_{\mathrm{loc}}(\mathbb R)$,
\begin{align}
\lim_{k\to\infty}\int f(x-d^{-k}t)\ d\nu(t)=f(x)
\end{align}
for Lebesgue-almost every $x$.
\end{cor}

Our arguments are sufficiently robust to address lacunary differentiation at self-similar scales with respect to non-uniform Cantor measures as well; because this question is less motivated by the ergodic setting, we leave the details to the interested reader, see Remark \ref{r:rem} below.

\medskip

Two open problems naturally present themselves. First, our arguments are entirely $L^2$-based; it would be interesting to establish a convergence result for functions with less integrability. A second question concerns polynomial images of Cantor sets, i.e.\ the pointwise convergence of
\[
\frac1{|\mathcal C_N|}\sum_{n\in\mathcal C_N}f(T^{P(n)}x),
\]
where $P\in\mathbb Z[\cdot]$ is a polynomial with integer coefficients, and $T$ is invertible (or $P(\mathbb{N}) \subset \mathbb{N}$). While it seems plausible that work of \cite{Green} could be used to address the case where $P$ is a monomial, the general case remains out of reach; we look forward to addressing it in future work.

\bigskip

The structure of the paper is as follows: in Section \ref{s:Prelim}, we collect preliminary notation and recall L\'epingle's inequality; in Section \ref{s:var}, we establish our analytic tools; and in Section \ref{s:proof}, we prove Theorems \ref{t:mainZ}, \ref{t:main}, and \ref{t:main0} in turn. We conclude with Section \ref{s:euc}, where we establish Corollary \ref{c:cantor}.

\subsection{Acknowledgements}
A. I. was supported in part by the National Science Foundation under NSF DMS-2154232. B. K. was supported in part by an ERC Starting Grant. 

\section{Preliminaries}\label{s:Prelim}

\subsection{General Notation}
Throughout, we abbreviate the complex character
\[
e(t):=e^{2\pi it};
\]
for $f\in\ell^1(\mathbb Z)$, our Fourier convention is
\begin{align}\label{e:fourier-convention}
\widehat f(\beta):=\sum_{x\in\mathbb Z}f(x)e(-x\beta),
\qquad
F^{\vee}(x)=\int_{\mathbb T} F(\beta)e(x\beta)\ d\beta,
\end{align}
and for integrable functions on a probability space $(\Omega,\mathcal B,\mathbb P)$, we write $\mathbb E  F$ for expectation and $\mathbb E  (F\mid\mathcal B')$ for conditional expectation with respect to a sub-$\sigma$-algebra $\mathcal B'\subseteq\mathcal B$.

\subsection{Asymptotic Notation}\label{sss:O}
We will make use of the modified Vinogradov notation. We use $X \lesssim Y$ or $Y \gtrsim X$ to denote
the estimate $X \leq CY$ for an absolute constant $C$ and $X, Y \geq 0.$  If we need $C$ to depend on a
parameter, we shall indicate this by subscripts, thus for instance $X \lesssim_p Y$ denotes the estimate $X \leq C_p Y$ for some $C_p$ depending on $p$. We use $X \approx Y$ as shorthand for $Y \lesssim X \lesssim Y$. We also make use of big-Oh: we let $O(Y)$  denote a quantity that is $\lesssim Y$, and similarly
$O_p(Y )$ will denote a quantity that is $\lesssim_p Y$.

\subsection{Cantor-Set Notation}
Enumerate
\[
D=\{\delta_0=0<\delta_1<\cdots<\delta_{q-1}\};
\]
for
\[
m=\sum_{j=0}^Jb_jq^j,
\qquad 0\leq b_j<q,
\]
we define
\begin{align}\label{e:Phi}
\Phi(m):=\sum_{j=0}^J\delta_{b_j}d^j,
\end{align}
so that $\Phi:\mathbb Z_{\geq0}\to\mathcal C$ is a bijection, which is also strictly increasing. Indeed, if $m<m'$ and $J$ is the largest base-$q$ digit at which they differ, then the contribution of that digit to $\Phi(m')-\Phi(m)$ is at least $d^J$, while all lower digits can contribute adversely by at most
\[
(d-1)\sum_{j=0}^{J-1}d^j=d^J-1.
\]

For sequence space functions, define
\begin{align}\label{e:AN}
A_Mf(x):=\frac1M\sum_{m=0}^{M-1}f(x-\Phi(m)),
\end{align}
and
\begin{align}\label{e:Bn}
B_n f(x):=\frac1{q^n}\sum_{a\in \mathcal{D}_n}f(x-a),
\qquad
\mathcal{D}_n:=\Big\{\sum_{j=0}^{n-1}a_jd^j:a_j\in D\Big\},
\end{align}
with $\mathcal{D}_0=\{0\}$. Since $\Phi(\{0,\dots,q^n-1\})=\mathcal{D}_n$,
\begin{align}\label{e:Aq=B}
A_{q^n}=B_n.
\end{align}

We set
\begin{align}\label{e:getc}
g:=\gcd\{a:a\in D\},
\qquad
h(t):=\sin(2\pi gt),
\end{align}
and introduce
\begin{align}\label{e:rho}
\mathcal Z:=\Big\{\frac jg:0\leq j<g\Big\}\subset\mathbb T, \qquad  \rho(t):=\operatorname{dist}_{\mathbb T}(t,\mathcal Z);
\end{align}
since $d\mathcal Z\subseteq\mathcal Z$, one has $\rho(dt)\leq d\rho(t)$.

We next define
\begin{align}\label{e:m}
m(t):=\frac1q\sum_{a\in D}e(-at),
\qquad
\overline a_D:=\frac1q\sum_{a\in D}a,
\end{align}
where the first term is the fundamental Fourier multiplier, and the second term is essentially $m'(0)$, which arises in Taylor expansion arguments. We also introduce
\begin{align}\label{e:Q}
Q(t):=1-|m(t)|^2 = \frac4{q^2}\sum_{\substack{a,b\in D\\a<b}}\sin^2(\pi(a-b)t),
\end{align}
which exactly vanishes on $\mathcal{Z}$, so that by compactness and Taylor expansion we have the approximation
\begin{align}\label{e:Qrho}
Q(t)\approx_D\rho(t)^2;
\end{align}
in particular $Q$ behaves well under dilations by $d$,
\begin{align}\label{e:QQ}
Q(dt)\lesssim_{d,D}Q(t).
\end{align}

We shall repeatedly use two elementary consequences of \eqref{e:Qrho}. If $a\in C^1(\mathbb T)$ vanishes on $\mathcal Z$, then
\begin{align}\label{e:first-vanishing}
|a(t)|^2\lesssim_{a,D}Q(t),
\end{align}
and if $a\in C^2(\mathbb T)$ and both $a$ and $a'$ vanish on $\mathcal Z$, then
\begin{align}\label{e:second-vanishing}
|a(t)|\lesssim_{a,D}Q(t);
\end{align}
these follow from Taylor expansion.

\subsection{L\'epingle's Inequality}
The principal martingale input is the following form of L\'epingle's inequality \cite{LE}; we shall always use it with $r>2$.

\begin{proposition}[L\'epingle]\label{p:LE}
Let $F\in L^2(\Omega,\mathcal{B},\mathbb P)$ and let $(\mathcal{B}_j)_{j=0}^N$ be a finite monotone family of sub-$\sigma$-algebras of $\mathcal{B}$. Then
\begin{align}
\Big\|\mathcal V^r\big(\mathbb E(F\mid\mathcal{B}_j):0\leq j\leq N\big)\Big\|_{L^2(\Omega)}
\lesssim\frac{r}{r-2}\|F\|_{L^2(\Omega)}.
\end{align}
\end{proposition}

\section{Variational Estimates for Random Fourier Multipliers}\label{s:var}
This section contains the analytic core of the argument; we begin with the following lemma, philosophically similar to \cite[Proposition 5.1]{B4}. 

\begin{lemma}\label{l:firstvar}
For $0\leq k\leq K$, let $Y_k$ be bounded, integer-valued, independent random variables, and put
\[
\mathbf{m}_k(\beta):=\mathbb E e(Y_k\beta).
\]
Suppose that $c_k(\beta,Y_k)\in L^2(\Omega)$ satisfy
\begin{align}\label{e:random-coeff-hyp}
\mathbb E c_k(\beta,Y_k)=0,
\qquad
\sum_{k=0}^K\mathbb E|c_k(\beta,Y_k)|^2\leq \Lambda^2
\end{align}
for almost every $\beta$. Set
\[
\mathbf{v}_k(\beta):=\mathbb E\big[c_k(\beta,Y_k)e(Y_k\beta)\big]
\]
and
\begin{align}\label{e:TJ-random}
T_Jf(x):=\int_{\mathbb T}\widehat f(\beta)
\Big(\sum_{k=0}^J \mathbf{v}_k(\beta)\prod_{\ell<k}\mathbf{m}_\ell(\beta)\Big)e(\beta x)\,d\beta.
\end{align}
Then, for every $r>2$,
\begin{align}\label{e:firstvar-bound}
\|\mathcal V^r(T_Jf:0\leq J\leq K)\|_{\ell^2}
\lesssim\frac{r}{r-2}\Lambda\|f\|_{\ell^2},
\end{align}
where the implicit constant is independent of $K$.
\end{lemma}

\begin{proof}
Following the lead of Bourgain, we adopt a martingale approach, using the (reverse) filtration given by 
\begin{align}\label{e:Bj}
\mathcal{B}_j:=\sigma(Y_j,\dots,Y_K),
\qquad
\mathcal{B}_{K+1}:=\{\emptyset,\Omega\}.
\end{align}

Abbreviating
\[
U_k(\beta,\omega):=e\Big(-\beta\sum_{\ell=k+1}^KY_\ell(\omega)\Big),
\]
the relevant function is
\begin{align}\label{e:random-F}
F(x,\omega)
:=\int_{\mathbb T}\widehat f(\beta) 
\Big( \sum_{k=0}^Kc_k(\beta,Y_k(\omega))U_k(\beta,\omega)\Big) e(\beta x)\,d\beta;
\end{align}
the key point is that for fixed $\beta$, the summands in brackets are pairwise orthogonal in $L^2(\Omega)$. Indeed, if $k<j$, then every factor in the cross term other than $c_k(\beta,Y_k)$ is independent of $Y_k$, so integration in $Y_k$ annihilates this term. Plancherel and \eqref{e:random-coeff-hyp} therefore yield the bound
\begin{align}\label{e:random-F-L2}
\mathbb{E} \sum_x |F(x,\omega)|^2
&=\int_{\mathbb T}|\widehat f(\beta)|^2 \cdot
\mathbb E\Big|\sum_{k=0}^Kc_k(\beta,Y_k)U_k(\beta)\Big|^2\,d\beta\\
&=\int_{\mathbb T}|\widehat f(\beta)|^2\sum_{k=0}^K\mathbb E|c_k(\beta,Y_k)|^2\,d\beta\\
&\leq \Lambda^2\|f\|_{\ell^2}^2.
\end{align}

We now apply L\'{e}pingle's inequality: if $\Delta_k$ denotes the $k$th summand in \eqref{e:random-F}, we can express the partial sums as conditional expectations, see \eqref{e:Bj},
\[
\mathbb E(F\mid\mathcal{B}_j)=\sum_{k=j}^K \Delta_k.
\]
Indeed, the terms with $k\geq j$ are measurable, and if $k<j$, conditioning integrates in $Y_k$ and uses the mean-zero nature of $\mathbb Ec_k=0$; the convergence of \eqref{e:random-F-L2} allows one to freely use Fubini.

If we reverse orientation by setting
\[
S_J:=\sum_{k=0}^J \Delta_k
=F-\mathbb E(F\mid\mathcal{B}_{J+1}),
\]
then by L\'epingle's inequality and \eqref{e:random-F-L2},
\begin{align}\label{e:random-S-var}
\| \|\mathcal V^r(S_J:0\leq J\leq K)\|_{L^2(\Omega)} \|_{\ell^2}
\lesssim\frac{r}{r-2}\Lambda\|f\|_{\ell^2}.
\end{align}

Next, let $Y:=\sum_{\ell=0}^KY_\ell$ and translate the entire sequence by the same random integer,
\[
H_J(x,\omega):=S_J(x+Y(\omega),\omega),
\]
noting that translation preserves the $\ell^2$ norm and pointwise variation. Since
\[
U_k(\beta,\omega)e(\beta(x+Y))
=e(\beta x)e\Big(\beta\sum_{\ell=0}^kY_\ell\Big),
\]
independence allows us to express
\[
\mathbb EH_J(x,\omega)=T_Jf(x);
\]
by Minkowski's inequality for the finite-dimensional $\ell^r$ norm, \[
\mathcal V^r(T_Jf(x):J)
\leq\mathbb E\mathcal V^r(H_J(x,\omega):J),
\]
so taking $\ell^2$ norms, applying Jensen, and using \eqref{e:random-S-var}, proves \eqref{e:firstvar-bound}.
\end{proof}

Next, for $k\geq0$, define
\begin{align}\label{e:dilate}
m_k(t):=m(d^kt),
\qquad
Q_k(t):=Q(d^kt),
\qquad
P_k(t):=\prod_{\ell=0}^{k-1}m_\ell(t),
\qquad P_0(t)=1;
\end{align}
in particular, the operator $B_k$ has multiplier $P_k$:
\begin{align}\label{e:symbol}
\widehat{B_kf}(t)=P_k(t)\widehat f(t);
\end{align}
explicitly
\begin{align}
    P_k(t) = \frac{1}{q^k} \sum_{a \in \mathcal{D}_k} e(- at).
\end{align}

Our next lemma will be the principal tool in the proof of Theorem \ref{t:mainZ}.

\begin{lemma}\label{l:var2}
Suppose $u_k:\mathbb T\to\mathbb C$ satisfy
\[
|u_k(t)|\leq \Lambda Q_k(t),
\]
and define
\begin{align}\label{e:defect-operator}
\mathcal R_Jf(x):=\int_{\mathbb T}\widehat f(\beta)
\Big(\sum_{k=0}^JP_k(\beta)u_k(\beta)\Big)e(\beta x)\,d\beta.
\end{align}
Then, for every $r>2$,
\begin{align}\label{e:defect-var-bound}
\|\mathcal V^r(\mathcal R_Jf:J\geq0)\|_{\ell^2}
\lesssim_{d,D}\frac{r}{r-2}\Lambda\|f\|_{\ell^2}.
\end{align}
\end{lemma}

\begin{proof}
We first split the summands into their even and odd parts: for $\varepsilon\in\{0,1\}$, put
\[
k_j:=2j+\varepsilon,
\]
and define
\begin{align}\label{e:vj}
v_j(t):=u_{k_j}(t)
\prod_{\substack{\ell<k_j\\ \ell \not \equiv k_j\  \mod 2}}m_\ell(t),
\end{align}
so that we can express
\begin{align}\label{e:parity-factorization}
P_{k_j}(t)u_{k_j}(t)=v_j(t)\prod_{i<j}m_{k_i}(t).
\end{align}
We will show by algebraic manipulation that Lemma \ref{l:firstvar} can be applied to these multipliers; without this splitting, the randomization below would lead to the generally unbounded quantity $\sum_k Q_k(t)$.

Let $\mathsf D_j$ be independent random variables uniformly distributed on $D$, and set
\[
Y_j:=-d^{k_j}\mathsf D_j,
\]
so that $\mathbb Ee(Y_jt)=m_{k_j}(t)$. If
\[
Z_j(t):=e(Y_jt)-m_{k_j}(t),
\]
then $\mathbb EZ_j=0$ and $\mathbb E|Z_j|^2=Q_{k_j}$; define
\begin{align}\label{e:cj-definition}
c_j(t,\mathsf D_j):=\frac{v_j(t)\overline{Z_j(t)}}{Q_{k_j}(t)} \cdot \mathbf{1}_{\{ t : Q_{k_j}(t) > 0 \}},
\end{align}
noting that the bound $|u_{k_j}|\leq \Lambda Q_{k_j}$ forces $v_j=0$ whenever $Q_{k_j}$ vanishes. Straightforward computation then yields
\begin{align}\label{e:cj-identities}
\mathbb Ec_j=0,
\qquad
\mathbb E[c_je(Y_j\cdot)]=v_j,
\qquad
\mathbb E|c_j|^2=
\begin{cases}
\frac{|v_j|^2}{Q_{k_j}},&Q_{k_j}>0,\\
0,&Q_{k_j}=0,
\end{cases}
\end{align}
so
\begin{align}\label{e:cj-first-bound}
\mathbb E|c_j|^2
\leq \Lambda^2Q_{k_j}
\prod_{\substack{\ell<k_j\\ \ell \not \equiv k_j\  \mod 2}}|m_\ell|^2.
\end{align}

For the odd subseries, $k_j=2j+1$, \eqref{e:QQ} and \eqref{e:cj-first-bound} give
\begin{align}\label{e:cj-odd-bound}
\mathbb E|c_j|^2
\lesssim_{d,D}\Lambda^2Q_{2j}\prod_{i<j}|m_{2i}|^2;
\end{align}
for the even subseries, the term $j=0$ satisfies $\mathbb E|c_0|^2\leq \Lambda^2Q_0\leq \Lambda^2$, while for $j\geq1$,
\begin{align}\label{e:cj-even-bound}
\mathbb E|c_j|^2
\lesssim_{d,D}\Lambda^2Q_{2j-1}\prod_{i=0}^{j-2}|m_{2i+1}|^2.
\end{align}
Since by telescoping, for any increasing sequence $n_0<\cdots<n_M$,
\begin{align}\label{e:defect-telescope}
\sum_{j=0}^MQ_{n_j}\prod_{i<j}|m_{n_i}|^2
&=\sum_{j=0}^M\Big(\prod_{i<j}|m_{n_i}|^2-\prod_{i\leq j}|m_{n_i}|^2\Big)\\
&=1-\prod_{j=0}^M|m_{n_j}|^2\leq1,
\end{align}
for either parity we obtain the bound
\begin{align}\label{e:cj-total}
\sum_{j\geq0}\mathbb E|c_j(t,\mathsf D_j)|^2\lesssim_{d,D}\Lambda^2;
\end{align}
Lemma \ref{l:firstvar}, together with \eqref{e:parity-factorization}, controls the variation of the partial sums of the even and odd subseries separately.

To recover the original order, write $a_k:=P_ku_k$ and set
\[
E_j:=\sum_{i=0}^ja_{2i},
\qquad
O_j:=\sum_{i=0}^ja_{2i+1},
\qquad
O_{-1}:=0.
\]
Then
\[
\sum_{k=0}^{2j}a_k=E_j+O_{j-1},
\qquad
\sum_{k=0}^{2j+1}a_k=E_j+O_j.
\]
If we define
\begin{align}
    \mathcal{R}_J^{E} f &:= \int_{\mathbb{T}} \widehat{f}(\beta) \Big( \sum_{2i \leq J} P_{2i}(\beta) u_{2i}(\beta) \Big) e(\beta x) \ d\beta, \qquad \text{and} \\
    \mathcal{R}_J^{O} f &:= \int_{\mathbb{T}} \widehat{f}(\beta) \Big( \sum_{2i +1 \leq J} P_{2i+1}(\beta) u_{2i+1}(\beta) \Big) e(\beta x) \ d\beta,
\end{align}
then apply the triangle inequality
\begin{align}
    \| \mathcal{V}^r( \mathcal{R}_J f : J \geq 0) \|_{\ell^2} \leq     \| \mathcal{V}^r( \mathcal{R}_J^E f : J \geq 0) \|_{\ell^2} +     \| \mathcal{V}^r( \mathcal{R}_J^O f : J \geq 0) \|_{\ell^2}
\end{align}
and conclude, noting that repetition does not increase the variation, and that the initial contribution to $R_0^O$ is bounded by Plancherel,
\[ \| (\widehat{f} P_1 u_1)^{\vee} \|_{\ell^2} \leq \Lambda \|f \|_{\ell^2}.\]
\end{proof}

With Lemma \ref{l:var2}, we can quickly complete the main argument.

\section{Proofs of the Main Theorems}\label{s:proof}
We begin with complete blocks.

\begin{proposition}\label{p:lacprop}
For every $r>2$,
\begin{align}\label{e:lacprop}
\|\mathcal V^r(B_nf:n\geq0)\|_{\ell^2}
\lesssim_{d,D}\frac{r}{r-2}\|f\|_{\ell^2}.
\end{align}
\end{proposition}

To facilitate the proof, we define
\begin{align}\label{e:R}
R(t):=m(t)-1+i\frac{\overline a_D}{g}h(t),
\qquad
G(t):=dh(t)-m(t)h(dt);
\end{align}
since for every $z\in\mathcal Z$, one has $m(z)=1$, $h(z)=h(dz)=0$, and
\[
m'(z)=-2\pi i\overline a_D,
\qquad
h'(z)=h'(dz)=2\pi g,
\]
$R,R',G,G'$ vanish on $\mathcal Z$, and \eqref{e:second-vanishing} gives
\begin{align}\label{e:Rest}
|R(t)|+|G(t)|\lesssim_{d,D}Q(t).
\end{align}
These functions will arise in the course of certain telescoping arguments below; the only significant property of them that will be used is \eqref{e:Rest}.

\begin{proof}
It suffices to work with a finite truncation, uniformly in the terminal index. Set
\[
H_k(t):=h(d^kt),
\qquad
R_k(t):=R(d^kt),
\qquad
G_k(t):=G(d^kt),
\]
so that $|R_k|+|G_k|\lesssim Q_k$, while \eqref{e:first-vanishing} gives
\begin{align}\label{e:H-square-Q}
|H_k|^2\lesssim_DQ_k.
\end{align}
Since
\begin{align}\label{e:telescope}
P_{k+1}=P_km_k,
\end{align}
we have
\begin{align}\label{e:Pn-expansion}
P_n =1+\sum_{k=0}^{n-1}P_k(m_k-1) =1+\sum_{k=0}^{n-1}P_k \cdot \Big(-i\frac{\overline a_D}{g}H_k+R_k\Big),
\end{align}
see \eqref{e:getc}, \eqref{e:m}, and \eqref{e:R}, and Lemma \ref{l:var2} controls the partial sums with $R_k$. It remains to control
\[
\mathcal V^r\Big(\big (\widehat f \cdot \sum_{k=0}^JP_kH_k\big)^\vee:J\geq0\Big).
\]

Appealing to \eqref{e:telescope} again, one obtains the exact identity
\begin{align}\label{e:H-telescope}
(d-1)\sum_{k=0}^JP_kH_k
&=\sum_{k=0}^{J-1}P_k \cdot (dH_k-m_kH_{k+1})-H_0+dP_JH_J \\
& = \sum_{k=0}^{J-1}P_k \cdot G_k-H_0+dP_JH_J;
\end{align}
the first term is controlled by Lemma \ref{l:var2}, while the term $-H_0$ is independent of $J$, so it has zero variation; it remains to control the endpoint sequence $P_JH_J$.

But, by \eqref{e:H-square-Q},
\[
\sum_{J\geq0}|P_JH_J|^2
\lesssim_D\sum_{J\geq0}|P_J|^2Q_J,
\]
and since
\begin{align}\label{e:endpoint-telescope}
|P_J|^2Q_J
=|P_J|^2(1-|m_J|^2)
=|P_J|^2-|P_{J+1}|^2,
\end{align}
we may telescope
\[
\sum_{J\geq0}|P_J|^2Q_J
=\sum_{J\geq0}\bigl(|P_J|^2-|P_{J+1}|^2\bigr)
\leq1;
\]
dominating the variation by a square sum, and applying Plancherel, \eqref{e:endpoint-telescope} yields the bound
\[
\Big\|\mathcal V^r\big((\widehat fP_JH_J)^\vee:J\geq0\big)\Big\|_{\ell^2}
\lesssim_D\|f\|_{\ell^2},
\]
completing the proof.
\end{proof}

We next record the relevant maximal estimate needed for the convergence theorem.

\begin{cor}\label{c:all-prefix-max}
The following estimate holds:
\begin{align}\label{e:all-prefix-max-Z}
\Big\|\sup_{M\geq1}|A_Mf|\Big\|_{\ell^2}\lesssim_{d,D}\|f\|_{\ell^2}.
\end{align}
\end{cor}

\begin{proof}
Choose $n$ minimal with $M\leq q^n$. Then $M>q^{n-1}$, and the first $M$ values of $\Phi$ lie in $\mathcal{D}_n$. Therefore
\[
|A_Mf(x)|
\leq\frac1M\sum_{a\in \mathcal{D}_n}|f(x-a)|
\leq qB_n(|f|)(x),
\]
so
\[
\sup_M|A_Mf|\leq q\sup_nB_n(|f|)
\leq q\big(|f|+\mathcal V^r(B_n|f|:n\geq0)\big);
\]
Proposition \ref{p:lacprop} completes the proof.
\end{proof}


The fixed-mesh argument uses the following factorization.

\begin{lemma}\label{l:prefix-factorization}
Let
\[
\sigma_M(\beta):=\frac1M\sum_{m=0}^{M-1}e(-\beta\Phi(m))
\]
be the multiplier of $A_M$. Let $s\in\mathbb Z_{\geq0}$ and suppose that $q^s\leq u<q^{s+1}$. Then
\begin{align}\label{e:prefix-factorization}
\sigma_{uq^n}(\beta)=\sigma_u(d^n\beta)P_n(\beta).
\end{align}
In particular, $P_{n+s}(\beta) = P_n(\beta) P_s(d^n \beta)$.
\end{lemma}

\begin{proof}
Every $0\leq m<uq^n$ has a unique representation
\[
m=vq^n+r_0,
\qquad 0\leq v<u,
\quad 0\leq r_0<q^n;
\]
writing the base-$q$ expansions of $v$ and $r_0$, the two digit blocks occupy disjoint positions, so there are no carries and
\begin{align}\label{e:Phi-prefix}
\Phi(vq^n+r_0)=d^n\Phi(v)+\Phi(r_0).
\end{align}
Consequently,
\begin{align*}
\sigma_{uq^n}(\beta)
&=\frac1u\sum_{v=0}^{u-1}e(-d^n\beta\Phi(v))
\cdot\frac1{q^n}\sum_{r_0=0}^{q^n-1}e(-\beta\Phi(r_0))\\
&=\sigma_u(d^n\beta)P_n(\beta).
\end{align*}
\end{proof}

We are now prepared for the proof of our main result.
\begin{proof}[Proof of Theorem \ref{t:mainZ}]
Fix an integer $s\geq0$. For $q^s\leq u<q^{s+1}$, put
\[
b_{s,u}(\beta):=\sigma_u(\beta)-P_s(\beta).
\]
At every $z\in\mathcal Z$, all phases $e(-z\Phi(m))$ equal $1$, so $b_{s,u}(z)=0$, and since only finitely many $u$ occur for a fixed $s$, \eqref{e:first-vanishing} yields the bound
\begin{align}\label{e:bsu-Q}
|b_{s,u}(\beta)|^2\lesssim_{d,D,s}Q(\beta).
\end{align}
If we define the ``error'' multiplier by 
\begin{align}\label{e:mesh-error}
\mathcal E_{n,u}(\beta):=P_n(\beta)b_{s,u}(d^n\beta),
\end{align}
then by Lemma \ref{l:prefix-factorization},
\begin{align}\label{e:mesh-decomposition}
\sigma_{uq^n}(\beta)
=P_{n+s}(\beta)+\mathcal E_{n,u}(\beta),
\end{align}
while by \eqref{e:bsu-Q} we may bound
\begin{align}
\sum_{n\geq0}\sum_{q^s\leq u<q^{s+1}}|\mathcal E_{n,u}(\beta)|^2
&\lesssim_{d,D,s}\sum_{n\geq0}|P_n(\beta)|^2Q_n(\beta)\\
&=\sum_{n\geq0}\big(|P_n(\beta)|^2-|P_{n+1}(\beta)|^2\big)\\
&\leq1;
\end{align}
Plancherel then allows us to dispose of the error terms
\begin{align}\label{e:mesh-square}
\Big\|\Big(\sum_{n\geq0}\sum_{q^s\leq u<q^{s+1}}
|(\mathcal E_{n,u} \widehat{f})^{\vee}|^2\Big)^{1/2}\Big\|_{\ell^2}
\lesssim_{d,D,s}\|f\|_{\ell^2}.
\end{align}

Next, order the pairs $(n,u)$ lexicographically, first by $n$ and then by $u$; this agrees with the increasing numerical order of $uq^n$, since
\[
(q^{s+1}-1)q^n<q^sq^{n+1}.
\]
The averages $B_{n+s}f$ in \eqref{e:mesh-decomposition} are repeated a fixed finite number of times for each $n$, and repetitions do not increase variation; since for any scalar family $(z_{n,u})$ and $r\geq2$,
\[
\mathcal V^r(z_{n,u}:(n,u)\text{ lexicographic})^2
\lesssim \sum_{n,u}|z_{n,u}|^2,
\]
Proposition \ref{p:lacprop}, \eqref{e:mesh-decomposition}, and \eqref{e:mesh-square} yield the bound
\[
\|\mathcal V^r(A_Mf:M\in\mathcal G_s)\|_{\ell^2}
\lesssim_{d,D,s}\frac{r}{r-2}\|f\|_{\ell^2},
\]
completing the proof.
\end{proof}

Theorem \ref{t:main} follows from Theorem \ref{t:mainZ} by Calder\'on's transference principle.

\begin{proof}[Proof of Theorem \ref{t:main0}]
We first prove convergence of the averages $A_M^T$, see \eqref{e:ergodic-count-prefix}.

Fix an integer $s\geq0$. Theorem \ref{t:main} implies that $A_M^Tf$ converges almost everywhere as $M\to\infty$ through $\mathcal G_s$, for every $f\in L^2(X)$; by density, appropriately transferring Corollary \ref{c:all-prefix-max}, we may therefore restrict to $f\in L^2(X)\cap L^\infty(X)$ below. So, for every sufficiently large $M$, choose $n\geq0$ so that
\[
q^{n+s}\leq M<q^{n+s+1},
\]
and approximate $M$ by elements in $\mathcal{G}_s$ by defining
\[
u:=\Big\lfloor\frac{M}{q^n}\Big\rfloor,
\qquad
\pi_s(M):=uq^n\in\mathcal G_s.
\]
Then
\[
0\leq M-\pi_s(M)<q^n\leq q^{-s}\pi_s(M),
\]
and since $\pi_s(M)\leq M$,
\begin{align}\label{e:mesh-approximation}
|A_M^Tf-A_{\pi_s(M)}^Tf|
\lesssim \frac{M-\pi_s(M)}{M}\|f\|_{L^\infty(X)}
\lesssim q^{-s}\|f\|_{L^\infty(X)};
\end{align}
since the sequence $A_{\pi_s(M)}^T f$ converges almost everywhere for each $s$, taking $M,M'\to\infty$ in \eqref{e:mesh-approximation} gives
\[
\limsup_{M,M'\to\infty}|A_M^Tf(x)-A_{M'}^Tf(x)|
\lesssim q^{-s}\|f\|_{L^\infty(X)},
\]
and letting $s\to\infty$ proves convergence of the full sequence $A_M^Tf(x)$ for $f\in L^2(X)\cap L^\infty(X)$. 

Finally, let
\[
M(N):=|\mathcal C_N|;
\]
since $\Phi$ is strictly increasing and $\Phi(0)=0$,
\[
\mathcal C_N
=\{\Phi(1),\dots,\Phi(M(N))\},
\]
so
\[
C_N^Tf
=\frac{M(N)+1}{M(N)}A_{M(N)+1}^Tf-\frac1{M(N)}f.
\]
As $N\to\infty$, one has $M(N)\to\infty$, so convergence of the re-parametrized averages implies convergence of $C_N^Tf$; we extend this to all $L^p(X)$ functions, $2 \leq p < \infty$, by transferring Corollary \ref{c:all-prefix-max}.

\end{proof}

We conclude the paper by applying the above in the Euclidean context.

\section{Euclidean Analogues}\label{s:euc}
Once again, let $(\mathsf D_j)_{j\geq1}$ be independent random variables, each uniform on $D$, and set
\begin{align}\label{e:Y}
Y:=\sum_{j\geq1}\mathsf D_jd^{-j}
\end{align}
Let $\nu$ be the law of $Y$, the natural Cantor measure on $\mathcal C'$, see \eqref{e:cantor'}; $\nu$ is continuous because the masses of cylinder sets decay geometrically. Define
\begin{align}\label{e:Euclidean-operator}
\mathcal K_kf(x):=\int f(x-d^{-k}y)\,d\nu(y);
\end{align}
Corollary \ref{c:cantor} then follows directly from the following proposition and a standard localization argument.

\begin{proposition}\label{p:Euclidean-variation}
For every $r>2$,
\begin{align}\label{e:Euclidean-variation}
\|\mathcal V^r(\mathcal K_kf:k\geq0)\|_{L^2(\mathbb R)}
\lesssim_{d,D}\frac{r}{r-2}\|f\|_{L^2(\mathbb R)}.
\end{align}
\end{proposition}

\begin{proof}
First truncate the variation to $0\leq k\leq K$. By density, it is enough to prove a bound independent of $K$ for compactly supported $d^{-L}$-step functions with $L>K$:
\begin{align}\label{e:disccts}
f(x)=\sum_{m\in\mathbb Z}F(m) \cdot \mathbf1_{[m/d^L,(m+1)/d^L)}(x);
\end{align}
note that in \eqref{e:disccts}
\begin{align}\label{e:step-norm}
\|f\|_{L^2(\mathbb R)}^2=d^{-L}\|F\|_{\ell^2(\mathbb Z)}^2,
\end{align}
and that $\mathcal{K}_k f(x)$ can be expressed as a convex combination of $B_{L-k} F(m)$ and $B_{L-k} F(m-1)$ whenever 
\[ x \in [m/d^L,(m+1)/d^L).\]

More precisely, fix $0\leq k\leq K$ and put $n=L-k$. Splitting the first $n$ digits from the tail yields the decomposition
\[
d^nY=J_n+Y',
\]
where $J_n$ is uniform on $\mathcal{D}_n$, $Y'$ has the same law as $Y$, and $J_n,Y'$ are independent. For $x \in [m/d^L,(m+1)/d^L)$, write
\[
x=d^{-L}(m+\vartheta),
\qquad
m\in\mathbb Z,
\quad 0\leq\vartheta<1,
\]
so that
\[
x-d^{-k}Y=d^{-L}(m-J_n+\vartheta-Y').
\]
Since $0\leq Y'\leq1$, this point lies in the cell indexed by $m-J_n$ when $Y'\leq\vartheta$, and in the cell indexed by $m-J_n-1$ when $Y'>\vartheta$, so
\begin{align}\label{e:coarse-identity}
&\mathcal K_kf(d^{-L}(m+\vartheta)) \\
& \qquad =\nu([0,\vartheta]) \cdot B_{L-k}F(m)+(1-\nu([0,\vartheta])) \cdot B_{L-k}F(m-1).
\end{align}
Since the variation seminorm is convex, and reversing the finite order $k=0,\dots,K$ merely reverses the corresponding indices $n=L,L-1,\dots,L-K$, \eqref{e:coarse-identity} and Proposition \ref{p:lacprop} yield the bound
\begin{align*}
\int_{\mathbb R}\mathcal V^r(\mathcal K_kf(x):0\leq k\leq K)^2\,dx
&\lesssim d^{-L}\sum_m\mathcal V^r(B_nF(m):n\geq0)^2\\
&\lesssim\Big(\frac{r}{r-2}\Big)^2d^{-L}\|F\|_{\ell^2}^2\\
&=\Big(\frac{r}{r-2}\Big)^2\|f\|_{L^2(\mathbb R)}^2.
\end{align*}
Since the constant is independent of $K$ and $L$, we conclude by density and monotone convergence.
\end{proof}

\begin{remark}\label{r:rem}
    The arguments developed above are sufficiently flexible to address the case of Cantor sets with non-uniform measures, i.e.\ those for which the i.i.d. random variables $(\mathsf D_j)_{j\geq1}$ are \emph{not} uniformly distributed on $D$. One analogously addresses the discrete convolution operators
    \begin{align}
        \sum_{a_0,\dots,a_{n-1}} \mathbb{P}(\mathsf D_j = a_j, \ 0 \leq j \leq n-1) \cdot f(x-\sum_{j=0}^{n-1} a_j d^j)
    \end{align}
    and suitably rescales; the key replacement is at the level of Fourier multipliers, in which case one substitutes
    \begin{align}
        m(t) \longrightarrow \mathbb{E} e(- t \mathsf D)
    \end{align}
so that
\begin{align}
    Q(t) \longrightarrow 4 \sum_{a < b} \mathbb{P}(\mathsf D_1 = a, \ \mathsf{D}_2 = b) \cdot \sin^2(\pi(a-b)t)
\end{align}
    still vanishes (to second order) precisely on $\mathcal{Z}$.
\end{remark}

\end{document}